\documentclass[11pt,a4paper,twoside]{amsart}

\usepackage[all,arrow,2cell]{xy}

\newcommand{\confer}{{\em cf.}\ }

\newcommand{\ca}{{\mathcal A}}
\newcommand{\cb}{{\mathcal B}}
\newcommand{\cc}{{\mathcal C}}
\newcommand{\cd}{{\mathcal D}}

\newcommand{\ch}{{\mathcal H}}

\newcommand{\ck}{{\mathcal K}}

\newcommand{\cs}{{\mathcal S}}
\newcommand{\ct}{{\mathcal T}}
\newcommand{\cu}{{\mathcal U}}
\newcommand{\cv}{{\mathcal V}}

\newcommand{\opname}[1]{\operatorname{\mathsf{#1}}}

\renewcommand{\mod}{\opname{mod}\nolimits}

\newcommand{\Inj}{\opname{Inj}\nolimits}
\newcommand{\Mod}{\opname{Mod}\nolimits}

\newcommand{\per}{\opname{per}\nolimits}

\newcommand{\tria}{\opname{tria}\nolimits}
\newcommand{\Tria}{\opname{Tria}\nolimits}
\newcommand{\Dif}{\opname{Dif}\nolimits}

\renewcommand{\ker}{\opname{ker}\nolimits}

\newcommand{\Hom}{\opname{Hom}}

\newcommand{\RHom}{\opname{RHom}}

\newcommand{\Ext}{\opname{Ext}}

\newcommand{\End}{\opname{End}}

\newcommand{\ten}{\otimes}
\newcommand{\lten}{\overset{\boldmath{L}}{\ten}}

\newtheorem{Thm}{Theorem}

\newtheorem{theorem}{Theorem}
\newtheorem*{theorem*}{Theorem}

\newtheorem{corollary}[Thm]{Corollary}

\newtheorem*{conjecture*}{Conjecture}

\newtheorem*{remark*}{Remark}

\newcommand{\ie}{{\em i.e.}}
\newcommand{\eg}{{\em e.g.}}

\title{Recollements from generalized tilting}

\author{Dong Yang}
\address{
Dong Yang\\Max-Planck-Institut f\"ur Mathematik in Bonn, Vivatsgasse
7, 53111 Bonn, Germany} \email{yangdong98@mails.thu.edu.cn}

\date{Last modified on \today.}

\begin{document}

\begin{abstract} Let $\ca$ be a small dg category over a field $k$ and let $\cu$ be
a small full subcategory of the derived category $\cd\ca$ which
generate all free dg $\ca$-modules. Let $(\cb,X)$ be a standard lift
of $\cu$. We show that there is a recollement such that its middle
term is $\cd\cb$, its right term is $\cd\ca$, and the three functors
on its right side are constructed from $X$. This applies to the pair
$(A,T)$, where $A$ is a $k$-algebra and $T$ is a good $n$-tilting
module, and we obtain a result of Bazzoni--Mantese--Tonolo. This
also applies to the pair $(\ca,\cu)$, where $\ca$ is an augmented dg
category and $\cu$ is the category of `simple' modules, \eg~ $\ca$
is a finite-dimensional algebra or the Kontsevich--Soibelman
$A_\infty$-category associated to a quiver with potential.
\end{abstract}

\maketitle

A \emph{recollement} of triangulated categories is a diagram of
triangulated categories and triangle functors
\[\xymatrix{\ct''\ar[rr]^{i_*}&&\ct\ar[rr]^{j^*}\ar@/^20pt/[ll]^{i^!}\ar@/_20pt/[ll]_{i^*}&&\ct'\ar@/^20pt/[ll]^{j_*}\ar@/_20pt/[ll]_{j_!}},\]
such that \begin{itemize} \item $(i^*,i_*,i^!)$ and $(j_!,j^*,j_*)$
are adjoint triples;
\item $i_*,j_*,j_!$ are fully faithful;
\item $j^*\circ i_*=0$;
\item for every object $X$ of $\ct$ there are two triangles
\[\xymatrix{i_*i^!\ar[r] & X \ar[r] & j_*j^*X\ar[r]&}\text{ and }\xymatrix{j_!j^*X\ar[r] & X \ar[r] & i_*i^*X\ar[r] &},\]
where the four morphisms are the units and counits.
\end{itemize}
We also say that this is a recollement of $\ct$ in terms of $\ct'$
and $\ct''$. This notion was introduced by
Beilinson--Bernstein--Deligne in \cite{BeilinsonBernsteinDeligne82} 
in geometric contexts, where stratifications of varieties induce
recollements of derived categories of sheaves.

In algebraic contexts, recollements are closely related to tilting
theory. Let $A$ be a ring. Let $\cd(A)=\cd(\Mod A)$ denote the
derived category of (right) $A$-modules, and $\per A$ denote the
triangulated subcategory of $\cd(A)$ generated by the free module of
rank $1$. An object $T$ of $\per A$ is called a \emph{partial
tilting complex} if $\Hom_{\cd(A)}(T,\Sigma^n T)=0$, and a
\emph{tilting complex} if in addition $\tria(T)=\per A$, where
$\tria(T)$ is the triangulated subcategory of $\cd(A)$ generated by
$T$. Rickard's Morita theorem for derived categories states that the
standard functors associated to a tilting complex $T$ over $A$ are
triangle equivalences between $\cd(A)$ and $\cd(\End_{\cd(A)}(T))$,
see~\cite{Rickard89}. Later in~\cite{Koenig91}, Koenig proved that
under certain conditions a partial tilting complex $T$ over $A$
yields a recollement of $\cd(A)$ in terms of $\cd(\End_{\cd(A)}(T))$
and a third derived category which measures how far the associated
standard functors are from being equivalences (see
also~\cite{Happel92}~\cite{Miyachi03}). In this sense, a recollement
of derived categories can be viewed as a natural generalization of a
derived equivalence. The relation between tilting theory and
recollements of derived categories has been further studied
in~\cite{AngeleriKoenigLiu10}~\cite{ChenXi10}. The dg version of
Rickard's theorem was developed by Keller in~\cite{Keller94}, and
the result of Koenig was generalized to the dg setting by
J{\o}rgensen~\cite{Jorgensen06} and
Nicol\'{a}s--Saor\'{\i}n~\cite{NicolasSaorin09},
where the role of partial tilting complexes is played by compact
objects.

In this paper we deal with a situation which is `dual' to the one
in~\cite{Koenig91}~\cite{Jorgensen06}~\cite{NicolasSaorin09}.
Starting from a dg category $\ca$ and a set of objets in the derived
category $\cd\ca$ which generates all the compact objects, we
construct a dg category $\cb$ together with a recollement of
$\cd\cb$ in terms of $\cd\ca$ and another derived category, see
Theorem~\ref{t:recollement}. We identify this third derived category
with a certain known category in the special case when $\ca$ is the
Kontsevich--Soibelman $A_\infty$-category associated to a quiver
with potential (Corollary~\ref{c:ginzburg-algebra}) or when $\ca$ is
a finite-dimensional self-injective algebra
(Corollary~\ref{c:self-injective-algebra}). The motivation for our
study was to have a better understanding of the `exterior' case of
the Koszul duality (Corollary~\ref{c:koszul-duality}) and a result
of Bazzoni--Mantese--Tonolo which says that the right derived
Hom-functor associated to an (infinitely generated) good tilting
module is fully faithful (Corollary~\ref{c:large-tilting}).

\section{The main result}

Let $k$ be a field and let $\ca$ be a small dg $k$-category. Denote
by $\Dif\ca$ the dg category of (right) dg $\ca$-modules. A dg
$\ca$-module $M$ is \emph{$\ck$-projective} if the dg functor
$\Dif\ca(M,?)$ preserves acyclicity.
For example, the free modules $A^{\wedge}=\Dif\ca(?,A)$, $A\in\ca$,
are $\ck$-projective. Let $\cd\ca$ denote the derived category of
$\ca$, which is triangulated with suspension functor $\Sigma$ the
shift functor. For a set of objects or a subcategory $\cs$ of
$\cd\ca$ we denote by $\tria\cs$ the smallest triangulated
subcategory of $\cd\ca$ containing all objects in $\cs$ and closed
under taking direct summands. Let
$\per\ca=\tria(A^{\wedge},A\in\ca)$. An object $M$ of $\cd\ca$ is
\emph{compact} if the functor $\cd\ca(M,?)$ commutes with infinite
(set-indexed) direct sums, or equivalently, if $M$ belongs to
$\per\ca$. See~\cite{Keller94}.

Let $\cu$ be a full small subcategory of $\cd\ca$ such that
\begin{eqnarray} \tria\cu\supseteq\per\ca. \label{condition}\end{eqnarray}
 Let $(\cb,X)$ be
a \emph{standard lift} of $\cu$ (\cite[Section 7]{Keller94}).
Precisely, $\cb$ is a dg subcategory of $\Dif\ca$ consisting of
$\ck$-projective resolutions over $\ca$ of objects of $\cu$ (to
avoid confusion, for each object $B$ of $\cb$ we will denote by
$U_B$ the corresponding dg $\ca$-module) and $X$ is the dg
$\cb^{op}\ten\ca$-module defined by $X(B,A)=U_B(A)$. It induces a
pair of adjoint dg functors and a pair of adjoint triangle functors
\[\xymatrix{\Dif\cb\ar@<.7ex>[r]^{T_X}&\Dif\ca\ar@<.7ex>[l]^{H_X},}\qquad \xymatrix{\cd\cb\ar@<.7ex>[r]^{\mathbf{L}T_X}&\cd\ca\ar@<.7ex>[l]^{\mathbf{R}H_X}.}\]
When $\ca$ and $\cb$ are dg $k$-algebras (\ie~dg $k$-categories with
one object), the functors $\mathbf{L}T_X$ and $\mathbf{R}H_X$ are
usually written as $?\lten X$ and $\RHom(X,?)$.

Let $X^T$ be the dg $\ca^{op}\ten\cb$-module defined by
\[X^T(A,B)=\Dif\ca(X^B,A^{\wedge}),\]
where for $B\in\cb$, $X^B$ is by definition the dg $\ca$-module
$X(B,?)$. From the definition of $X$ we see that $X^B=U_B$. The main
result of this paper is

\begin{theorem}\label{t:recollement} Assume notations as above.
There is a dg $k$-category $\cc$ and  a recollement of triangulated
categories
\[\xymatrix{\cd\cc\ar[rr]^{i_*}&&\cd\cb\ar[rr]^{j^*}\ar@/^20pt/[ll]^{i^!}\ar@/_20pt/[ll]_{i^*}&&\cd\ca\ar@/^20pt/[ll]^{j_*}\ar@/_20pt/[ll]_{j_!}},\]
where the adjoint triple $(i^*,i_*,i^!)$ is defined by a dg functor
$F:\cb\rightarrow\cc$ (which is bijective on objects) such that
$i_*=F^*:\cd\cc\rightarrow\cd\cb$ is the pull-back functor, and the
adjoint triple $(j_!,j^*,j_*)$ is given by
\begin{eqnarray*}
j_!&=&\mathbf{L}T_{X^T},\\
j^*&=&\mathbf{R}H_{X^T}\hspace{5pt}\simeq\hspace{5pt}\mathbf{L}T_X,\\
j_*&=&\mathbf{R}H_X.
\end{eqnarray*}
\end{theorem}


\begin{proof}
In view of~\cite[Theorem 5]{NicolasSaorin09},
it suffices to prove
\begin{itemize} \item[(a)] $\mathbf{L}T_{X^T}$ is fully faithful,
\item[(b)]$\mathbf{R}H_{X^T}\simeq\mathbf{L}T_X$.
\end{itemize}

The proof for (a) is the same as the proof of~\cite[Lemma 10.5 the
`exterior' case c)]{Keller94}. Since $(\cb,X)$ is a lift, the
restriction of $\mathbf{L}T_X$ on the perfect derived category
$\per\cb$ is fully faithful, and its essential image is $\tria\cu$
(see~\cite[Section 7.3]{Keller94}):
\[\xymatrix{\mathbf{L}T_X|_{\per\cb}:\per\cb\ar[r]^(0.62){\sim}&\tria\cu.}\]
It is clear that $\mathbf{R}H_X$ takes an object of $\tria\cu$ into
$\per\cb$. Therefore, the restriction $\mathbf{R}H_X|_{\tria\cu}$ is
a quasi-inverse of $\mathbf{L}T_X|_{\per\cb}$, and hence is fully
faithful. It follows from~\cite[Lemma 6.2 a)]{Keller94} that the
restriction $\mathbf{L}T_{X^T}|_{\per\ca}$ is naturally isomorphic
to the restriction of $\mathbf{R}H_X|_{\per\ca}$, which is fully
faithful by condition (\ref{condition}). Condition (\ref{condition})
also implies that $\mathbf{R}H_X(A^{\wedge})=(X^T)^A$ ($A\in\ca$)
belongs to $\per\cb$, and hence is compact by~\cite[Theorem
5.3]{Keller94}. Now applying~\cite[Lemma 4.2 b)]{Keller94}, we
obtain that $\mathbf{L}T_{X^T}$ is fully faithful, finishing the
proof of (a).

Let $Y\rightarrow X^T$ be a $\ck$-projective resolution of dg
$\ca^{op}\ten\cb$-modules. Then the specialization $Y^A\rightarrow
(X^T)^A$ is a $\ck$-projective resolution of dg $\cb$-modules for
any object $A$ of $\ca$. Recall that $(X^T)^A$ is compact.
 It follows from~\cite[Lemma 6.2 a)]{Keller94} that
$\mathbf{L}T_{Y^T}\simeq \mathbf{R}H_Y$.
By~\cite[Lemma 6.1 b)]{Keller94}, in order to prove
$\mathbf{R}H_{X^T}\simeq\mathbf{L}T_X$, it suffices to prove that as
dg $\cb^{op}\ten \ca$-modules $Y^T$ and $X$ are quasi-isomorphic.
Let $A\in\ca$ and $B\in\cb$. We have $H_X(U_B)=B^{\wedge}$, and
hence
\begin{eqnarray*}
Y^T(A,B)&=&\Dif\cb(Y^A,B^{\wedge})\\
&=&\Dif\cb(Y^A,H_X(U_B))\\
&\cong&\Dif\ca(T_X(Y^A),U_B). \end{eqnarray*} The composition
$T_X(Y^A)\rightarrow T_X((X^T)^A)=T_X\circ H_X(A^\wedge)\rightarrow
A^{\wedge}$ is exactly the counit
$\mathbf{L}T_X\circ\mathbf{R}H_X(A^{\wedge})\rightarrow A^{\wedge}$,
which is an isomorphism in $\cd\ca$ because the restriction of
$\mathbf{R}H_X$ on $\per\ca$ is fully faithful. Moreover, both
$T_X(Y^A)$ and $A^{\wedge}$ are $\ck$-projective dg $\ca$-modules.
Therefore we have
\begin{eqnarray*}Y^T(A,B)
&\stackrel{q.is}{\leftarrow}&\Dif\ca(A^{\wedge},U_B)\\
&=&U_B(A)\\
&=&X(A,B).
\end{eqnarray*}
Further, every morphism in the above is functorial in both $A$ and
$B$. This completes the proof of (b).
\end{proof}


\begin{corollary}[\cite{BazzoniManteseTonolo09}]\label{c:large-tilting}
Let $A$ be a $k$-algebra, and $n$ be a positive integer. Let $T$ be
a \emph{good $n$-tilting module}, \ie~ $T$ is an $A$-module such
that
\begin{itemize}
\item[(T1)] the
projective dimension of $T$ is less than or equal to $n$;
\item[(T2)] $\Ext_A^i(T,T^{(\alpha)})=0$ for any integer $i>0$ and for any
cardinal $\alpha$;
\item[(T3)] there is an exact sequence
\[\xymatrix{0\ar[r] & A\ar[r] & T^0\ar[r] & T^1\ar[r] &\ldots \ar[r] & T^n\ar[r] & 0},\]
where $T^0,\ldots,T^n$ are direct summands of direct sums of finite
copies of $T$.
\end{itemize} Put $B=\End_A(T)$. Then the right derived functor
$\RHom_A(T,?):\cd(A)\rightarrow\cd(B)$ is fully faithful, and
$\cd(A)$ is triangle equivalent to the triangle quotient of $\cd(B)$
by the kernel of the left derived functor $?\lten_B T$.
\end{corollary}
\begin{proof} Let $\cu$ be the full subcategory of $\cd(A)$ consisting of one object $T$.
Then the condition (T3) implies the condition (\ref{condition}). Let
$X$ be a projective resolution of $T$ over $B^{op}\ten_k A$, and let
$\tilde{B}$ be the dg $k$-algebra $\Dif A(X,X)$. Then $X$ is
$\ck$-projective over $A$, and $(\tilde{B},X)$ is a standard lift of
$T$.  Thanks to (T2), the representation map $B\rightarrow\tilde{B}$
of the dg $B$-$A$-bimodule $X$ is a quasi-isomorphism, inducing
mutually quasi-inverse triangle equivalences
$?\lten_{\tilde{B}}\tilde{B}=\RHom_{\tilde{B}}(\tilde{B},?):\cd(\tilde{B})\rightarrow\cd(B)$
and $?\lten_B \tilde{B}:\cd(B)\rightarrow\cd(\tilde{B})$. Now
applying Theorem~\ref{t:recollement} and composing the resulting
recollement with the above triangle equivalences, we obtain a
recollement
\[\xymatrix{\cd(C)\ar[rrr]&&&\cd(B)\ar[rrr]^{\RHom_B(X^T,?)\simeq ?\lten_B
X}\ar@/^30pt/[lll]\ar@/_30pt/[lll]&&&\cd(A)\ar@/^30pt/[lll]^{\RHom_A(X,?)}\ar@/_30pt/[lll]_{?~\lten_A
X^T}},\] where $C$ is a dg $k$-algebra. Since $X$ and $T$ are
quasi-isomorphic as dg $B^{op}\ten_k A$-modules, we have natural
isomorphisms $?\lten_B X\simeq ~?\lten_B T$ and $\RHom_A(X,?)\simeq
\RHom_A(T,?)$ (\cite[Lemma 6.1 b)]{Keller94}). The desired result
follows at once.
\end{proof}

\begin{remark*}
\begin{itemize}
\item[a)] This result is due to Bazzoni~\cite{Bazzoni09} for $n=1$ and
Bazzoni--Mantese--Tonolo~\cite{BazzoniManteseTonolo09} for general
$n$ for all rings $A$.
\item[b)] By Theorem~\ref{t:recollement}, the left half of the recollement in
the proof is induced from a dg homomorphism $B\rightarrow C$. For
the case $n=1$ and for all rings $A$ Chen--Xi obtained
in~\cite{ChenXi10} such a recollement with $C$ being an ordinary
ring (so that the map $B\rightarrow C$ becomes a homomorphism of
rings). They used some results in~\cite{AngeleriKoenigLiu10} and
many other results such as the homological properties of the tilting
module $T$.
\end{itemize}
\end{remark*}

To state the next corollary, we need to introduce some notions. Let
$\ca$ be an \emph{augmented} dg $k$-category (\cite[Section
10.2]{Keller94}), \ie
\begin{itemize}
\item distinct objects of $\ca$ are non-isomorphic,
\item for each $A\in\ca$, a dg module $\bar{A}$ is given such that
$H^0\bar{A}(A)\cong k$ and $H^n\bar{A}(A')$ whenever $n\neq 0$ or
$A'\neq A$.
\end{itemize}
Let $(\ca^*,X)$ be a standard lift of
$\cu=\{\bar{A}|A\in\ca\}\subset\cd\ca$. By abuse of language, we
call the dg $k$-category $\ca^*$ the \emph{Koszul dual} of $\ca$.
Assume that the condition (\ref{condition}) holds, \eg~ this happens
in the `exterior' case in~\cite[Section 10.5]{Keller94}.

\begin{corollary}\label{c:koszul-duality}
Assume notations as above. There is a recollement of derived
categories of dg $k$-categories
\[\xymatrix{\cd\cc\ar[rrr]&&&\cd\ca^*\ar[rrr]^{\mathbf{R}H_{X^T}\simeq \mathbf{L}T_X}\ar@/^20pt/[lll]\ar@/_20pt/[lll]&&&\cd\ca\ar@/^20pt/[lll]^{\mathbf{R}H_X}\ar@/_20pt/[lll]_{\mathbf{L}T_{X^T}}}.\]
\end{corollary}
\begin{proof}
This is a direct consequence of Theorem~\ref{t:recollement}.
\end{proof}

\section{The left term}
As in the preceding section, we let $k$ be a field, $\ca$ be a small
dg $k$-category, $\cu$ be a full small subcategory of the derived
category $\cd\ca$ such that $\tria\cu\supseteq \per\ca$, and let
$(\cb,X)$ be a standard lift of $\cu$. Theorem~\ref{t:recollement}
says that there is a recollement of $\cd\cb$ in terms of $\cd\ca$
and a third derived category $\cd\cc$, where $\cc$ is a dg
$k$-category whose objects are in bijection with the objects of
$\cu$.

Let $\cv=\{(X^T)^A|A\in\ca\}\subset\cd\cb$. From the proof of the
theorem we obtain a commutative diagram
\[\xymatrix{\mathbf{R}H_X|_{\tria\cu}:&\tria\cu\ar[rr]^{\sim} && \per\cb\\
\mathbf{R}H_X|_{\per\ca}:&\per\ca\ar[rr]^{\sim}\ar@{^{(}->}[u] &&
\tria\cv.\ar@{^{(}->}[u]}\] Therefore $\mathrm{R}H_X$ induces a
triangle equivalence between the triangle quotient categories
\[\xymatrix{\tria\cu/\per\ca\ar[rr]^{\sim}&&\per\cb/\tria\cv}.\]
For a triangulated category $\ct$, let $\ct^c$ denote the
subcategory of compact objects in $\ct$. Let $\Tria\cv$ be the
localizing subcategory of $\cd\cb$ generated by the objects in
$\cv$. We have $(\cd\cb)^c=\per\cb$, and $(\Tria\cv)^c=\tria\cv$.
Thus by~\cite[Theorem 2.1]{Neeman92a}, the category
$(\cd\cb/\Tria\cv)^c$ is triangle equivalent to the idempotent
completion of $\per\cb/\tria\cv$. Since the essential image of
$\mathbf{L}T_{X^T}$ is exactly $\Tria\cv$, it follows that $\cd\cc$
is triangle equivalent to the triangle quotient $\cd\cb/\Tria\cv$,
and hence is an `unbounded version' of
$\tria\cu/\per\ca\cong\per\cb/\tria\cv$. Apparently, $\cd\cc$
vanishes if and only if so does $\tria\cu/\per\ca$, in which case
$\cu$ consists of a set of compact generators for $\cd\ca$.

In the following two special cases, we are able to identify $\cd\cc$
with a certain known category (however, the dg category $\cc$ is not
easy to describe).

\begin{corollary}\label{c:ginzburg-algebra}
Let $(Q,W)$ be a quiver with potential. Let $\ca_{(Q,W)}$ be the
Kontsevich--Soibelman $A_\infty$-category (\cite[Section
3.3]{KontsevichSoibelman08}) (or its enveloping dg category), let
$\widehat{\Gamma}_{(Q,W)}$ be the complete Ginzburg dg category
(\cite[Section 5]{Ginzburg06}), and let $\widetilde{C}_{(Q,W)}$ be
the `unbounded version' of the generalized cluster category
(\cite[Remark 4.1]{KellerYang10}). Then there is a recollement of
triangulated categories

\[\xymatrix{\widetilde{C}_{(Q,W)}\ar[rr]&&\cd\widehat{\Gamma}_{(Q,W)}\ar[rr]\ar@/^20pt/[ll]\ar@/_20pt/[ll]&&\cd\ca_{(Q,W)}\ar@/^20pt/[ll]\ar@/_20pt/[ll]}.\]
\vspace{0pt}
\end{corollary}
\begin{proof}
Let $\ca=\ca_{(Q,W)}$ and let $\cu$ be the category of simple
$\ca$-modules. Then condition (\ref{condition}) holds since $\ca$ is
finite-dimensional, and there is a standard lift $(\cb,X)$ such that
the dg category $\Gamma=\widehat{\Gamma}_{(Q,W)}$ (as the Koszul
dual of $\ca_{(Q,W)}$) is quasi-isomorphic to $\cb$. By
Corollary~\ref{c:koszul-duality}, there is a recollement with the
middle term being $\cd\Gamma$, the right term being $\cd\ca$, and
the right upper functor being $\mathbf{L}T_{X^T}$. It remains to
prove that the left term of this recollement is triangle equivalent
to $\widetilde{\cc}_{(Q,W)}$. Object sets of $\ca$, of $\cu$, and of
$\Gamma$ can all be identified with the vertice set $Q_0$ of the
quiver $Q$. For a vertex $i$ of $Q$, considered as an object of
$\ca$, the right dg $\Gamma$-module $(X^T)^i$ is isomorphic in
$\cd\Gamma$ to $\Sigma^{-3}S_i$, where $S_i$ the simple top of the
free
$\Gamma$-module $i^{\wedge}$. 
Thus the essential image of $\mathbf{L}T_{X^T}$ is the localizing
subcategory $\cd_0\Gamma=\Tria(S_i,i\in Q_0)$ of $\cd\Gamma$
generated by the $S_i$, $i\in Q_0$. Thus the left term of the
recollement is triangle equivalent to the triangle quotient
$\cd\Gamma/\cd_0\Gamma$, which is by definition
$\widetilde{\cc}_{(Q,W)}$.
\end{proof}

Let $A$ be a finite-dimensional basic $k$-algebra. Let $S$ be the
direct sum of the objects in a set of representatives of isomorphism
classes of simple $A$-modules, and let $X$ be a projective
resolution of $S$. Then $A^*=\Dif A(X,X)$ is the Koszul dual of $A$.

\begin{corollary}[\cite{Krause05}]\label{c:self-injective-algebra}
Let $A$ be a finite-dimensional basic self-injective $k$-algebra,
and $A^*$ its Koszul dual. Let $\underline{\Mod}A$ be the stable
category of the category $\Mod A$ of $A$-modules. Then there is a
recollement of triangulated categories

\[\xymatrix{\underline{\Mod}A\ar[rr]&&\cd(A^*)\ar[rr]\ar@/^20pt/[ll]\ar@/_20pt/[ll]&&\cd(A)\ar@/^20pt/[ll]\ar@/_20pt/[ll]}.\]
\vspace{0pt}
\end{corollary}

\begin{proof} Let $\mod A$ be the category of finite-dimensional $A$-modules,
and $\underline{\mod}A$ its stable category. As a triangulated
subcategory of $\cd(A)$, the bounded derived category $\cd^b(\mod
A)$ of $\mod A$ coincides with $\tria S$. Recall that the essential
image of $?\lten_A X^T$ is $\Tria X^T$. Consider the following
commutative diagram
\[\xymatrix{\Mod A\ar@{^{(}->}[r] & \cd(A) \ar[rr]^{\RHom_A(X,?)} && \cd(A^*)\ar[r] & \cd(A^*)/\Tria X^T\\
\mod A\ar@{^{(}->}[r]\ar@{^{(}->}[u] & \cd^b(\mod
A)\ar@{^{(}->}[u]\ar[rr]^{\sim} && \per A^* \ar@{^{(}->}[u]\ar[r] &
\per A^*/\tria X^T\ar@{^{(}->}[u]\\
& \per A \ar@{^{(}->}[u]\ar[rr]^{\sim} && \tria X^T\ar@{^{(}->}[u] &
}\] where the leftmost horizontal functors are the canonical
embeddings, and the rightmost horizontal functors are the canonical
projections. The restriction of $\RHom_A(X,?)$ on $\Mod A$ commutes
with infinite direct sums, because $X$ can be chosen such that its
component in each degree is a finitely generated projective
$A$-module. Therefore the composition of the three functors in the
first row, denoted by $F$, commutes with infinite direct sums. Since
$\RHom_A(X,A)\cong X^T$ belongs to $\Tria X^T$, it follows that $F$
factors through the stable category $\underline{\Mod}A$. In this
way, we obtain a triangle functor
\[\bar{F}:\underline{\Mod}A\rightarrow \cd(A^*)/\Tria X^T,\]
which commutes with infinite direct sums. It is known that
$\underline{\Mod}A$ is compactly generated by $\underline{\mod}A$
and $(\underline{\Mod}A)^c=\underline{\mod}A$. Moreover,  the
restriction $\bar{F}|_{\underline{\mod}A}$ is the composition of the
following three functors
\[\xymatrix{\underline{\mod}A\ar[r] & \cd^b(\mod A)/\per A\ar[r]^{\sim} & \per A^*/\tria X^T\ar@{^{(}->}[r] & \cd(A^*)/\Tria X^T}.\]
The first functor is also an equivalence (\cite[Theorem
2.1]{Rickard89b}). Therefore $\bar{F}$ induces a triangle
equivalence between $\underline{\mod}A=(\underline{\Mod}A)^c$ and
$\per A^*/\tria X^T=(\cd(A^*)/\Tria X^T)^c$. By~\cite[Lemma
4.2]{Keller94}, $\bar{F}$ itself is an equivalence. Now applying
Corollary~\ref{c:koszul-duality} we obtain the desired recollement.
\end{proof}

\begin{remark*} Let $\ch(\Inj A)$ be the homotopy category of
injective $A$-modules and $\ch_{ac}(\Inj A)$ be its full subcategory
of acyclic complexes. Applying a result of Krause~\cite[Corollary
4.3]{Krause05} to the algebra $A$, we obtain a recollement of
$\ch(\Inj A)$ in terms of $\cd(A)$ and $\ch_{ac}(\Inj A)$ with the
right middle functor being the canonical projection $Q:\ch(\Inj
A)\rightarrow \cd(A)$. We claim that this recollement is equivalent
to the one in Corollary~\ref{c:self-injective-algebra}. Indeed,
Krause proved in~\cite{Krause05} that $\ch(\Inj A)$ is compactly
generated by (an injective resolution of) the $A$-module $S$, and
that there is a triangle equivalence $\Theta:\ch(\Inj
A)\rightarrow\cd(A^*)$ taking $S$ to $A^*$. Since both
$\Theta(?)\lten_{A^*} X$ and $Q$ commute with infinite direct sums
and $\Theta(S)\lten_{A^*} X\cong X\cong S$, it follows that they are
isomorphic. Namely, the right middle parts of the two recollements
are equivalent via the equivalence $\Theta$. Therefore the two
recollements are equivalent.

Now let us construct the equivalence $\Theta$ by sketching the proof
of the assertion that $\ch(\Inj A)$ and $\cd(A^*)$ are triangle
equivalent. Let $\Dif_{\Inj}A$ be the full dg subcategory of $\Dif
A$ consisting of complexes of injective $A$-modules, let
$\mathbf{i}S$ be an injective resolution of the $A$-module $S$, and
put $B=\Dif A(\mathbf{i}S,\mathbf{i}S)$. Then the dg $B^{op}\ten
A^*$-module $\Dif A(X,\mathbf{i}S)$ yields a triangle equivalence
$\Phi:\cd(B)\rightarrow\cd(A^*)$ (see~\cite[Section 7.3]{Keller94}).
Moreover, $\Dif_{\Inj}A$ is a dg enhancement of the triangulated
category $\ch(\Inj A)$ in the sense of
Bondal--Kapranov~\cite{BondalKapranov90}, and there is a dg functor
$\Dif_{\Inj} A(\mathbf{i}S,?):\Dif_{\Inj} A\rightarrow\Dif B$.
Taking zeroth comhomologies gives us a triangle functor $\ch(\Inj
A)\rightarrow\ch(B)$, and composing it with the canonical projection
$\ch(B)\rightarrow\cd(B)$ we obtain a triangle equivalence
$\Psi:\ch(\Inj A)\rightarrow\cd(B)$ (\confer the proof
of~\cite[Theorem 4.3]{Keller94}). Now the composition
$\Theta=\Phi\circ\Psi:\ch(\Inj A)\rightarrow\cd(A^*)$ is a triangle
equivalence, which takes $\mathbf{i}S$ to $A^*$ (up to isomorphism),
as desired.
\end{remark*}

\noindent{ACKNOWLEDGEMENT.} The author gratefully acknowledges
financial support from Max-Planck-Institut f\"ur Mathematik in Bonn.
He thanks Pedro Nicol\'as and the referee for some helpful remarks
on a previous version. Part of the paper was written during the
author's visit to Department of Mathematics at Shanghai Jiaotong
University, he thanks Guanglian Zhang for his warm hospitality.



\def\cprime{$'$}
\providecommand{\bysame}{\leavevmode\hbox
to3em{\hrulefill}\thinspace}
\providecommand{\MR}{\relax\ifhmode\unskip\space\fi MR }
\providecommand{\MRhref}[2]{%
  \href{http://www.ams.org/mathscinet-getitem?mr=#1}{#2}
} \providecommand{\href}[2]{#2}


\end{document}